\documentclass{elsarticle}
\usepackage{amsmath}
\usepackage{amsthm}
\usepackage{amssymb}
\usepackage{tikz}

\newcommand{\R}{\mathbb{R}}  
\newcommand{\Sp}{\mathbb{S}} 
\newcommand{\iC}{\mathrm{i}} 
\newcommand{\Eu}{\mathrm{e}} 

\newcommand{\abs}[1]{\left|#1\right|}
\newcommand{\norm}[1]{\left\|#1\right\|}
\newcommand{\tonde}[1]{\left(#1\right)}
\newcommand{\graffe}[1]{\left\{#1\right\}}
\newcommand{\insieme}[2]{\graffe{#1 \colon #2}}

\newcommand{\Abs}[1]{\bigl|#1\bigr|}
\newcommand{\Norm}[1]{\bigl\|#1\bigr\|}
\newcommand{\Tonde}[1]{\bigl(#1\bigr)}

\newcommand{\TOnde}[1]{\Bigl(#1\Bigr)}

\renewcommand{\ge}{\geqslant} %
\renewcommand{\le}{\leqslant} %
\newcommand{\les}{\lesssim} 

\renewcommand{\d}{\,{\rm d}}

\newcommand{\der}[3][]{{\frac{{\rm d}^{#1} #2}{{\rm d}{#3}^{#1}}}}

\newcommand{\uno}{\mathbf{1}}
\newcommand{\pvector}[1]{%
  \begin{pmatrix}
    #1
  \end{pmatrix}}
\newcommand{\ddirac}[1]{%
  \,\boldsymbol{\delta}\!\pvector{#1}\!}

\newcommand{\1}[1]{\frac{1}{#1}}
\newcommand{\ovl}{\overline}
\newcommand{\mR}{{\mathcal R}}
\newcommand{\conv}{\ast}
\newcommand{\tensor}{\otimes}

\theoremstyle{plain}
\newtheorem{theorem}{Theorem}[section]
\newtheorem{lemma}[theorem]{Lemma}
\newtheorem{proposition}[theorem]{Proposition}
\newtheorem{corollary}[theorem]{Corollary}
\theoremstyle{definition}
\newtheorem{definition}[theorem]{Definition}
\theoremstyle{remark}
\newtheorem{remark}[theorem]{Remark}

\begin{document}

\title{Global Maximizers for the Sphere Adjoint Fourier Restriction Inequality}
\author{Damiano Foschi}
\ead{damiano.foschi@unife.it}
\address{Dipartimento di Matematica e Informatica, Universit\`a di Ferrara \\
  via Macchiavelli 35, 44121 Ferrara - Italy}
\date{October 26, 2013}

\begin{abstract}
  We show that constant functions are global maximizers
  for the adjoint Fourier restriction inequality for the sphere.
\end{abstract}



\maketitle

\section{Introduction}

Recently, Christ and Shao \cite{CS, CS2} have proved the existence of maximizers
for the adjoint Fourier restriction inequality of Stein and Tomas~\cite{St}
for the sphere:
\begin{equation} \label{eq:aFriS}
  \Norm{\widehat{f \sigma}}_{L^4(\R^3)} \les \norm{f}_{L^2(\Sp^2)}, 
\end{equation}
where $\Sp^2 = \insieme{x \in \R^3}{\abs{x} = 1}$ is the standard unit sphere
equipped with its natural surface measure~$\sigma$
induced by the Lebesgue measure on $\R^3$.
Here the Fourier transform of a integrable function~$f$ supported on the sphere is defined
for any~\mbox{$x \in \R^3$} by
\begin{equation*}
  \widehat{f \sigma}(x) = \int_{\Sp^2} \Eu^{-\iC x \cdot \omega} f(\omega) \d\sigma_\omega.
\end{equation*}
Let us denote by $\mR$ the optimal constant in \eqref{eq:aFriS}:
\begin{equation*}
  \mR := \sup_{f \in L^2(\Sp^2), f \neq 0}
  \frac{\Norm{\widehat{f \sigma}}_{L^4(\R^3)}}{\norm{f}_{L^2(\Sp^2)}}.
\end{equation*}
In \cite{CS},
using concentration compactness methods,
they prove that there exist sequences $\graffe{f_k}$
of nonnegative even functions in $L^2(\Sp^2)$ which converge to some maximizer
of the ratio $\|\widehat{f \sigma}\|_{L^4} / \norm{f}_{L^2}$,
but they do not compute the exact value of $\mR$.
Nevertheless, they show that constant functions are \emph{local} maximizers
and raise the question of whether constants are actually \emph{global} maximizers.
The purpose of this note is to give a positive answer to that question:

\begin{theorem} \label{th:main}
  A nonnegative function $f \in L^2(\Sp^2)$ is a global maximizer for~\eqref{eq:aFriS}
  if and only if it is a non zero constant,
  and we have
  \begin{equation*}
    \mR = \frac{\Norm{\widehat{\uno\sigma}}_{L^4(\R^3)}}{\norm{\uno}_{L^2(\Sp^2)}} = 2 \pi.
  \end{equation*}
\end{theorem}

When we combine Theorem~\ref{th:main} with the results of \cite[Theorem 1.2]{CS2}
we obtain that \emph{all} complex valued global maximizers for \eqref{eq:aFriS}
are of the form
\begin{equation*}
  f(\omega) = k \Eu^{\iC \theta} \Eu^{\iC \xi \cdot \omega},
\end{equation*}
for some $k > 0$, $\theta \in \R$, $\xi \in \R^3$.

A large part of the analysis carried out in \cite{CS} is local in nature
and it is based on a comparison between the case of the sphere and
that of a paraboloid which approximates the sphere at one point.
Here we are able to keep everything global, thanks to 
an interesting geometric feature of the sphere,
which is expressed in~Lemma~\ref{le:1234}.
It essentially says:
when the sum $\omega_1 + \omega_2 + \omega_3$ of three unit vectors
is again a unit vector, then we have
\begin{equation*}
  \abs{\omega_1 + \omega_2}^2 + \abs{\omega_1 + \omega_3}^2 + \abs{\omega_2 + \omega_3}^2  = 4.
\end{equation*}

In order to find maximizers for \eqref{eq:aFriS},
we follow the spirit of the proof of analogous results
obtained by the author for the paraboloid and the cone~\cite{Fo}.
The main steps are:
\begin{itemize}
\item The exponent $4$ is an even integer and 
  we can view the $L^4$ norm as a~$L^2$ norm of a product,
  which becomes, through the Fourier transform, a $L^2$ norm of a convolution.
  We write the $L^2$ norm of a convolution of measures supported on the sphere
  as a quadrilinear integral over a submanifold of~$(\Sp^2)^4$.
\item A careful application of the Cauchy-Schwarz inequality over that submanifold allows us
  to control the quadrilinear integral by some bilinear integral over $(\Sp^2)^2$.
\item Finally, by a spectral decomposition of the bilinear integral
  using spherical harmonics will show that the optimal bounds for the bilinear integral
  are obtained when we consider constant data.
\end{itemize}
We will see that every time an inequality appears,
the choice of $f$ constant will correspond to the case of equality.

\section{Quadrilinear form associated to the estimate}

\begin{definition}
  Given a complex valued function $f$ defined on $\Sp^2$,
  its \emph{antipodally conjugate} $f_\star$ is defined
  by $f_\star(\omega) := \ovl{f(-\omega)}$.
\end{definition}

By Plancherel's theorem we have
\begin{multline} \label{eq:A}
  \norm{\widehat{f \sigma}}_{L^4(\R^3)}^2 =
  \norm{\widehat{f \sigma} \ovl{\widehat{f \sigma}}}_{L^2(\R^3)} =
  \norm{\widehat{f \sigma} \widehat{f_\star \sigma}}_{L^2(\R^3)} = \\
  = \norm{\widehat{f \sigma \conv f_\star \sigma}}_{L^2(\R^3)} = 
  (2 \pi)^{\frac32} \Norm{f \sigma \conv f_\star \sigma}_{L^2(\R^3)}.
\end{multline}
When $f$ is constant we can explicitely compute this convolution.

\begin{lemma} \label{le:sigma}
  For $x \in \R^3$ we have
  \begin{equation*}
    \sigma \conv \sigma (x) =
    \iint_{(\Sp^2)^2} \ddirac{x - \omega - \nu} \d\sigma_\omega \d\sigma_\nu =
    \frac{2 \pi}{\abs{x}} \chi\Tonde{\abs{x} \le 2},
  \end{equation*}
  with norm $\norm{\sigma \conv \sigma}_{L^2(\R^3)} = 2^{5/2} \pi^{3/2}$.
\end{lemma}

The notation $\ddirac{\cdot}$ stands for
the Dirac's delta measure concentrated at the origin of~$\R^n$.

\begin{proof}
  The surface measure of the sphere can be written as
  \begin{equation*}
    \d\sigma_\omega = \ddirac{1 - \abs{\omega}} \d\omega =
    2 \ddirac{1 - \abs{\omega}^2} \d\omega.
  \end{equation*}
  The convolution then can be written as
  \begin{multline*}
    \sigma \conv \sigma (x) = 2 \int_{\Sp^2} \ddirac{1 - \abs{x - \omega}^2} \d\sigma_\omega =
    2 \int_{\Sp^2} \ddirac{2 x \cdot \omega - \abs{x}^2} \d\sigma_\omega = \\
    = \frac{2 \pi}{\abs{x}} \int_0^\pi \ddirac{\cos\theta - \frac{\abs{x}}2} \sin\theta \d\theta =
    \frac{2 \pi}{\abs{x}} \int_{-1}^1 \ddirac{c - \frac{\abs{x}}2} \d c 
    = \frac{2 \pi}{\abs{x}} \chi\tonde{\frac{\abs{x}}2 \le 1}.
  \end{multline*}
  The norm can then be easily computed,
  \begin{equation*}
    \norm{\sigma \conv \sigma}_{L^2(\R^3)}^2 = 4 \pi^2 \int_{\abs{x} \le 2} \frac{\d x}{\abs{x}^2} =
    4 \pi^2 4 \pi \int_0^2 \d r = 32 \pi^3.
  \end{equation*}
\end{proof}

For a generic data $f$, we can write the convolution in \eqref{eq:A} as
\begin{equation*}
  f \sigma \conv f_\star \sigma (x) = \iint_{\Sp^2 \times \Sp^2}  f(\omega) \ovl{f(-\nu)}
  \ddirac{x - \omega - \nu} \d\sigma_\omega \d\sigma_\nu.
\end{equation*}
The $L^2$ norm of the convolution can be written as a quadrilinear integral
\begin{multline} \label{eq:B}
  \Norm{f \sigma \conv f_\star \sigma}_{L^2(\R^3)}^2 = \\
  = \int_{(\Sp^2)^4} f(\omega_1) \ovl{f(-\nu_1)} \,\ovl{f(\omega_2)} f(-\nu_2)
  \ddirac{\omega_1 + \nu_1 - \omega_2 - \nu_2} 
  \d\sigma_{\omega_1} \d\sigma_{\nu_1} \d\sigma_{\omega_2} \d\sigma_{\nu_2} = \\
  = \int f(\omega_1) \ovl{f(-\omega_2)} f(\omega_3) \ovl{f(-\omega_4)} \d\Sigma_\omega =
  Q(f, f_\star, f, f_\star),
\end{multline}
where the measure $\Sigma$ is given by
\begin{equation} \label{eq:Sigma}
  \d\Sigma_{(\omega_1, \omega_2, \omega_3, \omega_4)} :=
  \ddirac{\omega_1 + \omega_2 + \omega_3 + \omega_4} 
  \d\sigma_{\omega_1} \d\sigma_{\omega_2} \d\sigma_{\omega_3} \d\sigma_{\omega_4},
\end{equation}
and $Q$ is the quadrilinear form defined by
\begin{equation} \label{eq:defQ}
  Q(f_1, f_2, f_3, f_4) :=
  \int_\Gamma f_1(\omega_1) f_2(\omega_2) f_3(\omega_3) f_4(\omega_4) \d\Sigma_\omega.
\end{equation}
Observe that $Q$ is fully symmetric in its arguments.

\begin{remark}
  The positive measure $\Sigma$ defined in \eqref{eq:Sigma}
  is supported on the (singular) submanifold~$\Gamma$ of $(\Sp^2)^4$ of (generic) dimension $5$
  given by
  \begin{equation*}
    \Gamma := \insieme{(\omega_1, \omega_2, \omega_3, \omega_4) \in (\Sp^2)^4}
    {\omega_1 + \omega_2 + \omega_3 +\omega_4 = 0}.
  \end{equation*}
  One way to visualize and parametrize $\Gamma$ is
  to choose freely the unit vectors~$\omega_1$ and~$\omega_2$,
  then~$\omega_3$ and~$\omega_4$ must be two diametrically opposite points on the circle
  obtained intersecting the unit sphere centered at~$0$
  with the unit sphere centered at~$-\omega_1-\omega_2$ (see Figure~\ref{fi:Gamma}).
  \begin{figure}[htbp]
    \centering

    \begin{tikzpicture}[scale = 3.2, >=latex]
      \coordinate (o) at (0,0);
      \coordinate (w1) at (-0.6,0.8);
      \coordinate (w2) at (-0.6,-0.8);
      \coordinate (w3) at (0.45,0.3);
      \coordinate (w4) at (0.75,-0.3);
      \coordinate (c) at (1.2,0);
      \coordinate (mw1) at (0.6,-0.8);
      \coordinate (mw2) at (0.6,0.8);
      \coordinate (m) at (0.6,0);
      \fill[fill opacity=0.25, ball color=white!05] (o) circle (1);
      \fill[fill opacity=0.08, ball color=white!05] (c) circle (1);
      \draw[thick,->] (o) -- (w1) node[left] {$\omega_1$};
      \draw[thick,->] (o) -- (w2) node[left] {$\omega_2$};
      \draw[thick,->] (o) -- (w3) node[right] {$\omega_3$};
      \draw[thick,->] (o) -- (w4) node[right] {$\omega_4$};
      \draw[dashed] (o) -- (mw1) node[below] {$-\omega_1$} -- (c);
      \draw[dashed] (o) -- (mw2) node[above] {$-\omega_2$} -- (c);
      \draw[dashed] (w3) -- (c);
      \draw[dashed] (w4) -- (c) node[right] {$-\omega_1-\omega_2 = \omega_3+\omega_4$};
      \draw (m) ellipse (0.16 and 0.8);
      \draw[dotted] (o) -- (c) (mw1) -- (mw2) (w3) -- (w4);
    \end{tikzpicture}    
    \caption{Parametrization of the manifold $\Gamma$} \label{fi:Gamma}
  \end{figure}
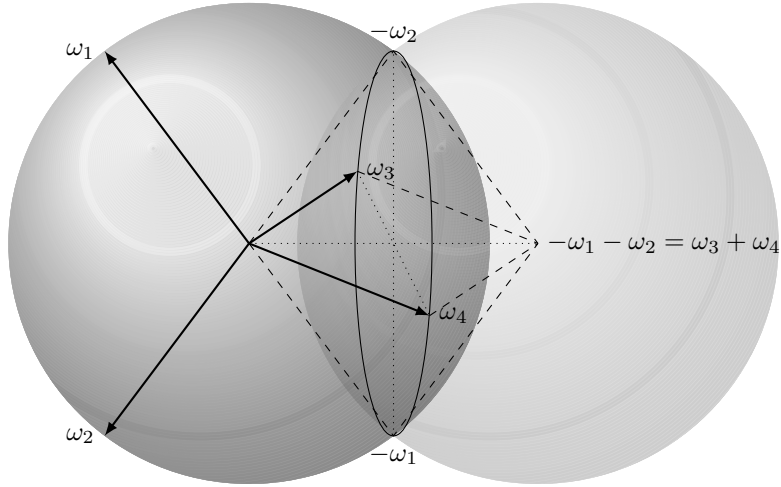
\end{remark}

\section{Symmetrization}

It is evident that
$\Abs{Q(f_1, f_2, f_3, f_4)} \le Q\Tonde{\abs{f_1}, \abs{f_2}, \abs{f_3}, \abs{f_4}}$,
with equality when the functions~$f_k$ are nonnegative.
Hence, we can reduce to consider nonnegative functions only.
We may say more.

\begin{definition}
  Given a complex valued function $f$ defined on $\Sp^2$
  we define its \emph{nonnegative antipodally symmetric rearrangement} $f_\sharp$ by
  \begin{equation*}
    f_\sharp(\omega) := \sqrt{\frac{\abs{f(\omega)}^2 + \abs{f(-\omega)}^2}2},
    \qquad \omega \in \Sp^2.
  \end{equation*}
The function $f_\sharp$ is also uniquely determined by the conditions
\begin{equation*}
  f_\sharp(\omega) = f_\sharp(-\omega) \ge 0, \qquad 
  f_\sharp(\omega)^2 + f_\sharp(-\omega)^2 = \abs{f(\omega)}^2 + \abs{f(-\omega)}^2
\end{equation*}
Moreover, we have $\norm{f_\sharp}_{L^2(\Sp^2)} = \norm{f}_{L^2(\Sp^2)}$.
\end{definition}

\begin{proposition} \label{pr:ff}
  We always have the pointwise estimate 
  \begin{equation} \label{eq:pointwise}
    \abs{f \sigma \conv f_\star \sigma (x)} \le f_\sharp \sigma \conv f_\sharp \sigma (x),
    \qquad \forall x \in \R^3.
  \end{equation}
\end{proposition}

By \eqref{eq:A} and~\eqref{eq:B} the proposition immediately implies:

\begin{corollary}[\cite{CS}] \label{co:ffff}
  We always have that
  \begin{equation*}
    Q(f, f_\star, f, f_\star) \le Q(f_\sharp, f_\sharp, f_\sharp, f_\sharp)
    \quad \text{and} \quad
    \Norm{\widehat{f \sigma}}_{L^4(\R^3)} \le \Norm{\widehat{f_\sharp \sigma}}_{L^4(\R^3)}.
  \end{equation*}
\end{corollary}

We also have equality when $f$ is a nonnegative constant function,
since in that case~\mbox{$f = f_\star = f_\sharp$}.
Corollary~\ref{co:ffff} was proved in~\cite{CS},
our proof here is much shorter and simpler.

\begin{proof}[Proof of Proposition \ref{pr:ff}]
  We may assume that $f$ is nonnegative.
  By the symmetry of the convolution,
  \begin{multline} \label{eq:fcf}
    2 f \sigma \conv f_\star \sigma (x) =
    f \sigma \conv f_\star \sigma (x) + f_\star \sigma \conv f \sigma (x) = \\
    = \int_{(\Sp^2)^2} \TOnde{f(\omega) f(-\nu) + f(-\omega) f(\nu)}
    \ddirac{x - \omega - \nu} \d\sigma_\omega \d\sigma_\nu.
  \end{multline}
  Now we use Cauchy-Schwarz inequality in its simplest form:
  \begin{equation} \label{eq:ABCD}
    A C + B D \le \sqrt{A^2 + B^2} \sqrt{C^2 + D^2},
  \end{equation}
  applied with $A = f(\omega)$, $B = f(-\omega)$, $C = f(-\nu)$, $D = f(\nu)$.
  We obtain
  \begin{equation*}
    f(\omega) f(-\nu) + f(-\omega) f(\nu) \le 2 f_\sharp(\omega) f_\sharp(\nu). 
  \end{equation*}
  We plug this into \eqref{eq:fcf} and obtain \eqref{eq:pointwise}.
\end{proof}

\begin{remark} \label{re:eqQ}
  When $A, B, C, D \ge 0$, we have equality in \eqref{eq:ABCD} if and only if~\mbox{$A D = B C$}.
  Suppose now that the equality~$Q(f, f_\star, f, f_\star) = Q(f_\sharp, f_\sharp, f_\sharp, f_\sharp)$
  holds for some nonnegative function $f$.
  It follows from the proof of Proposition \ref{pr:ff} that
  \begin{equation*}
    f(\omega) f(\nu) = f(-\omega) f(-\nu),
  \end{equation*}
  for almost every~$(\omega, \nu) \in (\Sp^2)^2$.
  If we integrate this identity with respect to~$\nu$,
  we obtain that~\mbox{$f(\omega) = f(-\omega)$} for almost every~$\omega \in \Sp^2$,
  which means that~$f = f_\star$ is antipodally symmetric.
\end{remark}

From now on, we may assume that $f = f_\sharp$ is a nonnegative antipodally symmetric function.

\section{Reduction to a quadratic form estimate}

Our goal now is to bound $Q(f,f,f,f)$ in terms of the $L^2$ norm of $f$.
We may try to use Cauchy-Schwarz inequality with respect to the measure $\Sigma$.

\begin{lemma} \label{le:BCS}
  Let $B(F, G)$ be the bilinear form given by
  \begin{equation*}
    B(F, G) = \int_\Gamma F(\omega_1, \omega_2) G(\omega_3, \omega_4) \d\Sigma_\omega,
  \end{equation*}
  for functions $F$ and $G$ defined on $\Sp^2 \times \Sp^2$.
  Then
  \begin{equation*}
    \Abs{B(F, G)}^2 \le B\Tonde{\abs{F}^2, \uno} B\Tonde{\abs{G}^2, \uno},
  \end{equation*}
  with equality if and only if there exist
  two constants $\lambda$, $\mu$ and a measurable function~$h(x)$
  defined on $\abs{x} \le 2$ such that
  \begin{equation*}
    F(\omega, \nu) = \lambda h(\omega + \nu), \quad
    G(\omega, \nu) = \mu h(-\omega -\nu), \quad
    \text{for almost every } \omega, \nu \in \Sp^2.
  \end{equation*}
\end{lemma}

\begin{proof}
  Apply Cauchy-Schwarz inequality
  to the product of~$F \tensor \uno$ and~$\uno \tensor G$
  with respect to the measure~$\Sigma$. 
  We have equality when~$F \tensor \uno$ and~$\uno \tensor G$ are linearly dependent
  on the support of~$\Sigma$.
  If~$F$ and $G$ are not identically zero,
  that happens when there are non zero constants~$\lambda$, $\mu$ such that
  \begin{equation*}
      \frac{F(\omega_1, \omega_2)}\lambda = \frac{G(\omega_3, \omega_4)}\mu =: h(x),
  \end{equation*}
  for almost every $\omega = \tonde{\omega_1, \omega_2, \omega_3, \omega_4} \in \Gamma$,
  with $x = \omega_1 + \omega_2 = - \omega_3 - \omega_4$.
\end{proof}

In our case $Q(f, f, g, g) = B(f \tensor f, g \tensor g)$.
Lemma~\ref{le:BCS} and Lemma~\ref{le:sigma} imply that
\begin{multline*}
  Q(f, f, f, f) \le
  Q\tonde{f^2, f^2, \uno, \uno} = \\
  = \iint_{(\Sp^2)^2} f(\omega_1)^2 f(\omega_2)^2
  \tonde{\iint_{(\Sp^2)^2} \ddirac{\omega_1 + \omega_2 + \omega_3 + \omega_4}
    \d\sigma_{\omega_3} \d\sigma_{\omega_4}} \d\sigma_{\omega_1} \d\sigma_{\omega_2} = \\
  = \iint_{(\Sp^2)^2} f(\omega_1)^2 f(\omega_2)^2 \frac{2 \pi}{\abs{\omega_1 + \omega_2}}
  \d\sigma_{\omega_1} \d\sigma_{\omega_2},
\end{multline*}
but unfortunately the last integral is too singular for our purposes.

The next lemma contains the geometric information about the symmetries of
the support of the measure~$\Sigma$
which allows us to neutralize the singularity of the previous integral.

\begin{lemma} \label{le:1234}
  Let $\omega_1, \omega_2, \omega_3, \omega_4 \in \Sp^2$
  be such that $\omega_1 + \omega_2 + \omega_3 + \omega_4 = 0$.
  Then
  \begin{equation*}
    \abs{\omega_1 + \omega_2} \abs{\omega_3 + \omega_4} +
    \abs{\omega_1 + \omega_3} \abs{\omega_2 + \omega_4} +
    \abs{\omega_1 + \omega_4} \abs{\omega_2 + \omega_3} = 4.
  \end{equation*}
\end{lemma}

\begin{proof}
  Let $X := \omega_1 \cdot \omega_2 + \omega_1 \cdot \omega_3 + \omega_2 \cdot \omega_3$.
  We have $\omega_1 + \omega_2 + \omega_3 = -\omega_4 \in \Sp^2$.
  This implies that
  \begin{equation*}
    1 = \abs{\omega_4}^2 = \abs{\omega_1 + \omega_2 + \omega_3}^2 = 3 + 2 X.
  \end{equation*}
  Hence $X = -1$.
  Then
  \begin{equation*}
    \abs{\omega_1 + \omega_2}^2 + \abs{\omega_1 + \omega_3}^2 + \abs{\omega_2 + \omega_3}^2
    = 6 + 2 X = 4.
  \end{equation*}
  To conclude the proof it is enough to observe that 
  $\abs{\omega_j + \omega_k} = \abs{\omega_m + \omega_n}$
  whenever~$(j,k,m,n)$ is any permutation of $(1,2,3,4)$.
\end{proof}

We combine the result of Lemma \ref{le:1234} with the symmetry properties of $Q$
and obtain 
\begin{equation} \label{eq:C}
  Q(f,f,f,f) =
  \frac34 \int_\Gamma f(\omega_1) f(\omega_2) \abs{\omega_1 + \omega_2}
  f(\omega_3) f(\omega_4) \abs{\omega_3 + \omega_4} \d\Sigma_\omega
  = \frac34 B(F, F),
\end{equation}
where $F(\omega, \nu) := f(\omega) f(\nu) \abs{\omega + \nu}$.
We apply the Cauchy-Schwarz inequality of Lemma~\ref{le:BCS},
use again Lemma~\ref{le:sigma} and obtain
\begin{equation} \label{eq:D}
  B(F, F) \le B(F^2, \uno) =
  2 \pi \iint_{(\Sp^2)^2} f(\omega_1)^2 f(\omega_2)^2 \abs{\omega_1 + \omega_2}
  \d\sigma_{\omega_1} \d\sigma_{\omega_2}.
\end{equation}

\begin{remark} \label{re:equalD}
  We have equality in \eqref{eq:D} if and only if $f(\omega) f(\nu) = h(\omega + \nu)$
  for almost every~$(\omega, \nu) \in (\Sp^2)^2$ and 
  for some measurable function $h(x)$ defined on $\abs{x} \le 2$;
  this happens for example when $f$ is a constant function.
\end{remark}

At this point, since $\abs{\omega_1 + \omega_2} \le 2$,
we can immediately deduce the estimate 
\begin{equation} \label{eq:BFF}
B(F^2, \uno) \le 4 \pi \norm{f}_{L^2}^4,
\end{equation}
and hence prove the inequality \eqref{eq:aFriS},
but the constant is not the optimal one
and we will have strict inequality also for $f$ constant.

\section{Spectral decomposition of the quadratic form}

We consider now the quadratic functional
\begin{equation} \label{eq:defH}
  H(g) := \iint_{(\Sp^2)^2} \ovl{g(\omega)} g(\nu) \abs{\omega - \nu} \d\sigma_\omega \d\sigma_\nu,
\end{equation}
which is well defined, real valued and continuous on $L^1(\Sp^2)$.
It is easy to verify that
\begin{equation*}
  \abs{H(g_1) - H(g_2)} \le
  2 \tonde{\norm{g_1}_{L^1(\Sp^2)} + \norm{g_2}_{L^1(\Sp^2)}} \norm{g_1 - g_2}_{L^1(\Sp^2)}.
\end{equation*}

We want to show that the value of $H(g)$ does not decrease
when we replace~$g$ with a constant function with the same mean value.

\begin{theorem} \label{th:Hg}
  Let $g \in L^1(\Sp^2)$.
  Let $\mu = \1{4 \pi} \int_{\Sp^2} g(\omega) \d\sigma_\omega$
  be the mean value of $g$ on the sphere.
  Then
  \begin{equation*}
    H(g) \le H(\mu \uno) = \abs{\mu}^2 H(\uno).
  \end{equation*}
  Moreover, equality holds if and only if $g$ is constant.
\end{theorem}

By the continuity of $H$ on $L^1(\Sp^2)$,
it is enough to prove the theorem for functions in a dense subset of~$L^1(\Sp^2)$,
for example in the Hilbert space $L^2(\Sp^2)$.
When~\mbox{$g \in L^2(\Sp^2)$},
we consider the decomposition of $g$ as a sum of its spherical harmonics components.
A spherical harmonic $Y_k$ of degree $k$ is an eigenfunction of~$\Delta_{\Sp^2}$
corresponding to the eigenvalue $-k(k+1)$,
\begin{equation*}
  \Delta_{\Sp^2} Y_k = -k (k + 1) Y_k,
\end{equation*}
where $\Delta_{\Sp^2}$ stands for the Lapace-Beltrami operator on the sphere
acting on scalar functions.
Any function in $L^2(\Sp^2)$ can be expanded as a sum of orthogonal spherical harmonics
(see for example \cite[chapter IV]{SW}).

Spherical harmonics are related to Legendre polynomials.
The latter can be defined in terms of a generating function:
when $\abs{r} < 1$ and $\abs{t} \le 1$,
if we write the power series expansion
\begin{equation} \label{eq:genfun}
  \tonde{1 - 2 r t + r^2}^{-\12} = \sum_{k \ge 0} P_k(t) r^k,
\end{equation}
then, for any integer $k \ge 0$,
the coefficient $P_k(t)$ is the Legendre polynomial of degree $k$.
These polynomials form a complete orthogonal system in $L^2([-1,1])$
and we have
\begin{equation*}
  \int_{-1}^1 P_k(t)^2 \d t = \frac2{2k+1}.
\end{equation*}
We are going to need the following facts about spherical harmonics and Legendre polynomials.

\begin{lemma}[Funk-Hecke formula] \label{le:FH}
  Let $\phi$ be a continuous functions on $[-1,1]$
  and~$Y_k$ be a spherical harmonic of degree $k$.
  Then for any $\omega \in L^2(\Sp^2)$ we have
  \begin{equation*}
    \int_{\Sp^2} \phi(\omega \cdot \nu) Y_k(\nu) \d\sigma_\nu = 2 \pi \lambda_k Y_k(\omega),
  \end{equation*}
  where
  \begin{equation} \label{eq:lambdak}
    \lambda_k = \int_{-1}^1 \phi(t) P_k(t) \d t,
  \end{equation}
  and $P_k$ is the Legendre polynomial of degree $k$.
\end{lemma}

A proof of Lemma \ref{le:FH} and its generalization to higher dimensions
can be found in~\mbox{\cite[p.~247]{EMOTii}}.

\begin{lemma} \label{le:Pkkk}
  For any integer $k \ge 1$ we have
  \begin{equation} \label{eq:rec0}
    (2 k + 1) P_k(t) = \der{}{t} \Tonde{P_{k+1}(t) - P_{k-1}(t)}.
  \end{equation}
\end{lemma}

\begin{proof}
  Differentiate \eqref{eq:genfun} with respect to $r$,
  \begin{equation*}
    (t - r) \tonde{1 - 2 r t + r^2}^{-\frac32} = \sum_{k \ge 0} k P_k(t) r^{k-1}.
  \end{equation*}
  Multiply on both sides by $1 - 2rt + r^2$,
  \begin{equation*}
    (t - r) \sum_{k \ge 0} P_k(t) r^k = (1 - 2 r t + r^2)  \sum_{k \ge 0} k P_k(t) r^{k-1}.
  \end{equation*}
  From this identity, equate the coefficients which multiply the same power~$r^k$,
  for any~\mbox{$k \ge 1$}, and obtain Bonnet's recursion formula
  \begin{equation*}
    (2 k + 1) t P_k(t) = (k+1) P_{k+1}(t) + k P_{k-1}(t).
  \end{equation*}
  Differentiate with respect to $t$,
  \begin{equation} \label{eq:rec1}
    (2 k + 1) P_k(t) = (k+1) P_{k+1}'(t) - (2 k + 1) t P_k'(t) + k P_{k-1}'(t).
  \end{equation}

  Now, differentiate \eqref{eq:genfun} with respect to $t$,
  \begin{equation*}
    \tonde{1 - 2 r t - r^2}^{-\frac32} = \sum_{k \ge 1} P_k'(t) r^{k-1}.    
  \end{equation*}
  Again, multiply on both sides by $1 - 2rt + r^2$, and obtain
  \begin{equation*}
    \sum_{k \ge 0} P_k(t) r^k = (1 - 2 r t + r^2)  \sum_{k \ge 1} P_k'(t) r^{k-1}.
  \end{equation*}
  From this identity, equate the coefficients which multiply the same power~$r^k$,
  for any~\mbox{$k \ge 1$}, and obtain another recurrence formula,
  \begin{equation} \label{eq:rec2}
    P_k(t) = P_{k + 1}'(t) - 2 t P_k'(t) + P_{k-1}'(t).
  \end{equation}
  To end the proof,
  multiply~\eqref{eq:rec1} by~$2$ and subtract~\eqref{eq:rec2} multiplied by~$2 k +1$
  to get~\eqref{eq:rec0}.
\end{proof}

We also need to know the sign of the coefficients \eqref{eq:lambdak}
when~\mbox{$\phi(t) = \sqrt{2 - 2 t}$}.

\begin{lemma} \label{le:Lambdak}
  The integrals $\Lambda_k := \int_{-1}^1 \sqrt{2 - 2 t} \, P_k(t) \d t$
  are negative numbers for all~$k \ge 1$.
\end{lemma}

\begin{proof}
  Let $k \ge 1$.
  We use Lemma \ref{le:Pkkk} and integration by parts,
  \begin{equation} \label{eq:LAA}
    (2 k + 1) \Lambda_K =
    \int_{-1}^1 \sqrt{2 - 2 t} \Tonde{P_{k+1}'(t) - P_{k-1}'(t)} \d t =
    A_{k+1} - A_{k-1},
  \end{equation}
  where
  \begin{equation*}
    A_k := \int_{-1}^1 \frac{P_k(t)}{\sqrt{2 - 2 t}} \d t =
    \lim_{r \to 1} \int_{-1}^1 \frac{P_k(t)}{\sqrt{1 - 2 r t + r^2}} \d t.
  \end{equation*}
  The convergence of the limit follows from Lebesgue's dominated convergence theorem,
  since we can use the inequality $1 - 2 r t + r^2 \ge 2 r (1-t)$ to bound the denominator.
  From the generating function identity~\eqref{eq:genfun}
  and the orthogonality properties of Legendre polynomials we deduce that 
  \begin{equation*}
    A_k = \lim_{r \to 1} r^k \int_{-1}^1 P_k(t)^2 \d t = \frac2{2 k +1}.
  \end{equation*}
  This shows that the coefficients $A_k$ form a decreasing sequence,
  and by \eqref{eq:LAA} it follows that $\Lambda_k$ is negative for any $k \ge 1$.
\end{proof}

\begin{proof}[Proof of Theorem \ref{th:Hg}]
  When $g$ is a function in $L^2(\Sp^2)$,
  we decompose it into the sum
  \begin{equation*}
    g(\omega) = \sum_{k \ge 0} Y_k(\omega), 
  \end{equation*}
  where $Y_k$ is a spherical harmonic of degree $k$.
  In particular, the spherical harmonic component of~$f$ of degree~$0$
  is given by the constant function~$\mu \uno$,
  where~$\mu$ is the mean value of~$f$ on~$\Sp^2$.
  We have
  \begin{equation*}
    H(g) = \sum_{j,k \ge 0}
    \iint_{(\Sp^2)^2} \ovl{Y_j(\omega)} Y_k(\nu) \abs{\omega - \nu} \d\sigma_\nu \d\sigma_\omega.
  \end{equation*}

  By the Funk-Hecke formula of Lemma \ref{le:FH} we have that
  \begin{equation*}
    \int_{\Sp^2} \abs{\omega - \nu} Y_k(\nu) \d\sigma_\nu = 
    \int_{\Sp^2} \sqrt{2 (1 - \omega\cdot\nu)} Y_k(\nu) \d\sigma_\nu =
    2 \pi \Lambda_k Y_k(\omega),
  \end{equation*}
  where $\Lambda_k$ are the coefficients computed in Lemma \ref{le:Lambdak}.
  By the orthogonality properties of spherical harmonics we deduce that
  \begin{equation*}
    H(g) = 2 \pi \sum_{k \ge 0} \Lambda_k \norm{Y_k}_{L^2(\Sp^2)}^2 \le
    2 \pi \Lambda_0 \norm{Y_0}_{L^2(\Sp^2)}^2 = H(\mu\uno),
  \end{equation*}
  since we know by Lemma \ref{le:Lambdak} that $\Lambda_k < 0$ when~$k \ge 1$.
  Here we have equality if and only if~$Y_k \equiv 0$ for all~$k \ge 1$,
  which means that~$f = Y_0$ is a constant function.

  The case for a generic $g \in L^1(\Sp^2)$ follows by a density argument
  and by the continuity of~$H$ on $L^1(\Sp^2)$.
\end{proof}

\section{Constants are (the only real valued) maximizers}

We are now ready to put together all the steps we need
in order to prove estimate~\eqref{eq:aFriS} with its best constant.
From \eqref{eq:A}, \eqref{eq:B} and Corollary~\ref{co:ffff} we have
\begin{equation*}
  \Norm{\widehat{f \sigma}}_{L^4(\R^3)}^4 =
  (2 \pi)^3 \Norm{f \sigma \conv f \sigma}_{L^2(\R^3)}^2 =
  (2 \pi)^3 Q(f, f_\star, f, f_\star) \le (2 \pi)^3 Q(f_\sharp, f_\sharp, f_\sharp, f_\sharp),
\end{equation*}
where $Q$ was defined in \eqref{eq:defQ}.
By Remark \ref{re:eqQ}, when $f$ is a nonnegative function
we have equality here if and only if $f = f_\sharp$ is antipodally symmetric.

From \eqref{eq:C}, \eqref{eq:D} and the symmetry of $f_\sharp$, we get
\begin{multline*}
  (2 \pi)^3 Q(f_\sharp, f_\sharp, f_\sharp, f_\sharp) \le \frac34 (2 \pi)^4
  \iint_{(\Sp^2)^2} f_\sharp(\omega)^2 f_\sharp(\nu)^2 \abs{\omega + \nu}
  \d\sigma_\omega \d\sigma_\nu = \\
  = 12 \pi^4 \iint_{(\Sp^2)^2} f_\sharp(\omega)^2 f_\sharp(\nu)^2 \abs{\omega - \nu}
  \d\sigma_\omega \d\sigma_\nu =
  12 \pi^4 H(f_\sharp^2),
\end{multline*}
where $H$ was defined in \eqref{eq:defH}.
As observed in Remark \ref{re:equalD},
we have equality here when $f$ is constant.

The mean value of $f_\sharp^2$ on $\Sp^2$ is
\begin{equation*}
  \mu := \1{4\pi} \int_{\Sp^2} f_\sharp(\omega)^2 \d\sigma_\omega = \1{4 \pi} \norm{f}_{L^2(\Sp^2)}^2
\end{equation*}
By Theorem \ref{th:Hg} we have that
\begin{equation*}
  12 \pi^4 H(f_\sharp^2) \le 12 \pi^4 \mu^2 H(\uno) = \frac34 \pi^2 H(\uno) \norm{f}_{L^2(\Sp^2)}^4.
\end{equation*}
Here equality holds if and only if $f_\sharp$ is constant.
The value of $H(\uno)$ is easily computed:
\begin{multline*}
  H(\uno) = \iint_{(\Sp^2)^2} \abs{\omega - \nu} \d\sigma_\nu \d\sigma_\omega = 
  \iint_{(\Sp^2)^2} \sqrt{2 (1 - \omega \cdot \nu)} \d\sigma_\nu \d\sigma_\omega = \\
  = 4 \pi \cdot 2\pi \cdot \sqrt2 \int_{-1}^1 \sqrt{1 - t} \d t = \frac{64}3 \pi^2.
\end{multline*}
The chain of inequalities collected in this section gives us
$\Norm{\widehat{f \sigma}}_{L^4(\R^3)}^4 \le 16 \pi^4 \norm{f}_{L^2(\Sp^2)}^4$,
with equality if and only if $f = f_\sharp$ is constant.
This proves Theorem \ref{th:main}.

\section*{Acknowledgments}

The author is grateful to Nicola Visciglia for suggesting to look at \cite{CS}
and work on this problem, and for his helpful comments on the first draft,
and to Rupert Frank for a remark
which allowed to considerably simplify the proof of Theorem~\ref{th:Hg}.

\bibliographystyle{plain}

\end{document}